\documentclass[12pt]{amsart}
\usepackage{amssymb}
\usepackage{amsfonts}
\usepackage{amscd}
\usepackage[T1]{fontenc}

\usepackage{xy}
\input xy
\xyoption{all}
\usepackage{pdflscape}
\usepackage{hyperref}

\usepackage[english]{babel}

\let\myacute=\'

\def\<{\langle}
\def\>{\rangle}

\def \begindm {\begin{displaymath}}

\def \enddm {\end{displaymath}}

\newtheorem{thm}{Theorem}[section]
\newtheorem{lemma}{Lemma}[section]

\newtheorem{claim}{Claim}[section]

\numberwithin{equation}{section}

\long\def\symbolfootnote[#1]#2{\begingroup\def\thefootnote{\fnsymbol{footnote}}\footnote[#1]{#2}\endgroup}

\title[Solvability in groups with chain conditions]
{Solvability in groups with a chain condition on uniformly definable subgroups}
\author[J. Derakhshan]{Jamshid Derakhshan}
\address{St Hilda's College, University of Oxford, Cowley Place, Oxford OX4 1DY, and Mathematical Institute, Oxford, OX2 6GG, UK}
\email{derakhsh@maths.ox.ac.uk}

\begin{document}


\subjclass[2000]{Primary}

\begin{abstract} 
We prove a structure theorem for periodic locally soluble groups satisfying a chain condition on intersections of relatively uniformly 
definable subgroups using results from the theory of stable groups. The result in particular shows that these groups are soluble, thus giving a model-theoretic and much simpler proof of a special case of a theorem of Bryant and Hartley on the solubility of periodic locally soluble groups satisfying the minimal condition on centralizers.
\end{abstract}

\maketitle

\section{\bf Introduction}\label{sec-introduction}

A definable subgroup of a group $G$ is one which consists of all 
$g\in G$ such that $\phi(x,\bar a)$ holds for some first-order 
formula $\phi(\bar x,\bar y)$ from the language of group thryeo, and parameters $\bar a$ 
from $G$, and is denoted by $\phi(G,\bar a)$. 
If $H$ is a subgroup of $G$, then a relatively definable subgroup 
of $H$ is defined as the intersection of a definable subgroup of $G$ with $H$ 
(we also say definable relative to $H$). 
A family of subgroups $H_i$ of $G$ is called uniformly definable 
if for some formula $\phi(\bar x,\bar y)$ and parameters $\bar a_i$, 
each $H_i$ is defined by $\phi(\bar x,\bar a_i)$ with parameters $\bar a_i$. Similarly one 
defines uniformly relatively definable subgroups of the subgroup $H$. 

The chain condition on intersections of uniformly definable 
subgroups (which we call icc) is the condition that any chain of subgroups of the 
form $\bigcap_i \phi(G,\bar a_i)$ has length bounded by some function of 
$\phi(\bar x,\bar y)$. In we replace the definable groups $\phi(G,\bar a)$ 
with relatively definable subgroups, we obtain the chain 
condition on intersections of uniformly relatively definable 
subgroups (which we call sub-icc). 

An icc-group is defined as a group which satisfies the icc chain condition. A group 
isomorphic to a subgroup of an icc-group shall be called a sub-icc 
group. It turns out that a sub-icc group satisfies the 
sub-icc chain condition.

If $G$ is an icc-group and $H$ a 
subgroup, then 
we shall be considering relatively definable subgroups of $H$. This 
refers to the ambient icc supergroup $X$. If $D$ is 
a relatively definable subgroup of $H$, then a relatively definable 
subgroup of $D$ also is determined by the same icc supergroup, and 
indeed it is the intersection with $D$ of a definable subgroup of $X$.

The above conditions occur naturally in model theory of groups and infinite group theory. Sub-icc groups 
are mc-groups, namely, groups that satisfy the minimal 
condition on centralizers. This condition holds in linear groups (over an arbitrary field) and in many other (non-linear) groups.  

Stable groups, which are defined groups whose first-order theory is stable are icc by the 
Baldwin--Saxl condition (cf. [6]), thus subgroups
of stable groups (which are called substable groups) are sub-icc. Note that centralizers 
of a family of subsets of a group are uniformly definable. For details see [6]. 
Stable groups play an important role in stability theory and model theory.

Our main result is the following. Let us recall some definitions and notation. The Hirsch--Plotkin radical 
of a group $G$, denoted $HP(G)$, is the largest locally nilpotent normal subgroup. In 
every group the product of two locally nilpotent normal subgroups is 
locally nilpotent. Hence every group has a Hirsch--Plotkin radical.  
Let $\pi$ denote a set of primes. 
An element of finite order in a group is a $\pi$-element if the prime
divisors of its order are in $\pi$.  
A periodic group $G$ is a $\pi$-group if it consists of
$\pi$-elements. 
The maximal normal $\pi$-subgroup of a group $G$ is
denoted by $O_{\pi}(G)$.

\begin{thm} Let $G$ be a countable periodic locally soluble 
sub-icc group. Let $\Pi^*$ denote the set of odd primes. Then $G$ is 
soluble and moreover the following hold:
\itemize
\item There is a relatively definable soluble supergroup $D_1$ 
of the Hirsch--Plotkin radical of $G$ such that for any $p\in \Pi^*$, 
any $p$-subgroup of $G_1:=G/D_1$ satisfies the minimal condition on 
subgroups,
\item For any $p\in \Pi^*$, there is a normal subgroup $H_p$ containing 
the $p'$-group $O_{p'}(G_1)$ which has finite index in $G$ and such that 
$H_p/O_{p'}(G_1)$ is a divisible abelian group,
\item The intersection $I:=\bigcap_{p\in \Pi^*} H_p$ 
is a soluble group containing 
all the divisible subgroups of $G_1$ such that for all $p\in \Pi^*$, 
all the Sylow-$p$ subgroups of $G_1/I$ are finite,
\item There is a relatively definable soluble normal subgroup of $G_1$ 
containing $I$ such that the quotient $G_2:=G_1/D_2$ has a normal 
relatively definable soluble subgroup $J$ with the property that all 
the Sylow $2$-subgroups of $G_2/J$ are finite,
\item The sub-icc quotient $G_2/J$ is a virtually abelian group such that 
all its Sylow $p$-subgroups satisfy the minimal condition on subgroups 
for all primes $p$,
\item Any subgroup of a homomorphic image of $G_1$ whose 
Sylow $2$-subgroups satisfy the minimal condition is virtually 
abelian and $G$ is soluble.\end{thm}

In [1], Bryant and Hartley proved that a periodic locally soluble mc group is soluble. Their proof uses various works on finite and infinite 
groups, and a deep result of John Thompson on bounding the Fitting height of a finite soluble group. Our result gives a new and much simpler proof of the sub-icc case of the 
Bryant-Hartley theorem using the machinery of stable group theory. Our proof uses only the basic structure 
theory of locally finite groups and of stable groups. 

Our methods will hopefully initiate a study of non-periodic locally soluble stable (or sub-icc) groups, of which nothing is known at present. In [3] we proved
several structure theorems for locally nilpotent stable and sub-icc groups. The methods of this paper should hopefully give some ideas on the soluble analogue of those results. We pose two basic 
problems in this respect. 

Recall that a group is called hypercentral if it posses a transfinite central series of 
subgroups. In [3] it was shown that periodic locally nilpotent sub-icc groups 
are hypercentral. A group is called hyperabelian if it has a transfinite 
chain of subgroups each normal in the next and each successive quotient 
is abelian. 

\

{\bf Problem 1.} Is a locally soluble sub-icc group hyperabelian?

\

Positive solution would yield positive solution to the following using standard results:

\

{\bf Problem 2.} Does a locally soluble sub-icc 
group contain a largest locally soluble normal subgroup (i.e. a locally soluble radical)?

We remark that in contrast to nilpotency, product of two locally 
soluble normal subgroups need not be locally soluble.

\section{Sub-icc groups, local solubility, and local finiteness}

\begin{lemma} Let $G$ be a group.
\itemize
\item Suppose that $G$ is an icc group. Let $H$ be a definable
subgroup of $G$. Then the conjugate $H^g$ (where $g\in G$) and 
the normalizer $N_{G}(H)$ are definable. Let $A$ be an arbitrary
subset of $G$. Then the centralizer $C_G(A)$ is definable.
\item Suppose that $G$ is a sub-icc group. Let $H$ be a relatively definable
subgroup of $G$. Then the conjugate $H^g$
(where $g\in G$) and 
the normalizer $N_{G}(H)$ are also definable relative
to $G$. Let $A$ be a 
subset of $G$. Then the centralizer $C_H(A)$ is 
definable relative to $G$.\end{lemma}
\begin{proof} If $H$ is defined by $\phi(x,\bar a)$ for some 
parameter $\bar a$ and
formula $\phi(x,\bar y)$, then $H^g$ is 
defined by $\phi(gxg^{-1},\bar a)$. The definability of $N_G(H)$ is clear
for a definable subgroup $H$. The family of centralizers 
$C_G(A)=\bigcap_{a\in A} C_G(a)$, where $A$ varies over arbitrary
subsets of $G$, is uniformly definable since centralizers 
$C_G(a_1,\dots,a_n)$ of finite sets are uniformly defined by the formula
$\bigwedge_{i=1}^nxa_i=a_ix$ and any centralizer is the 
centralizer of boundedly many elements by the icc. This proves the
first part. 

Let $X$ denote the icc-supergroup
of $G$. We first show the relative definability
of $N_G(H)$. As $H$ is 
definable relative to $G$, there is a definable subgroup $D$ of $X$
such that $H=D\cap G$. 
Let $I:=\bigcap_{g\in N_G(H)} D^g$. As $D$ and its conjugates are 
definable, $I$ is equal
to a finite sub-intersection $D^{g_1} \cap \dots \cap D^{g_l}$, 
where $g_1,\dots g_l \in N_G(H)$, 
by the icc chain condition in $X$. Hence 
$\bigcap_{g\in N_G(H)} D^g$ and hence the 
normalizer $N_G(\bigcap_{g\in N_G(H)} D^g)$ 
is relatively definable.
\begin{claim} $N_G(H)=N_G(\bigcap_{g\in N_G(H)}
D^g)$.\end{claim}  
\begin{proof} 
Let $a \in G$ be in the normalizer of $\bigcap_{g\in
N_G(H)} D^g$. We show that $H^a=H$.

Let $h \in H$. Then 
$h \in H \leq \bigcap_{g\in N_G(H)} D^g$. Hence
$h^a \in \bigcap_{g\in N_G(H)} D^g \leq D$.
But $h^a \leq G$. So $h^a \in D\cap G=H$ as $H=D\cap G$. 
Hence $H^a \leq H$. Now let $h \in H$. Then 
$$h \in H \leq \bigcap_{g\in N_G(H)}
D^g \leq (\bigcap_{g\in N_G(H)} D^g)^a \leq D^a.$$
So for some $d \in D$ we have $h=a^{-1}da$. Hence
$d=aha^{-1}$.  Thus $d \in G$ which implies that 
$d \in D\cap G$.
Again as $H$ is definable relative to $G$, we deduce that
$d \in H$. Therefore $d^a \in H^a$. But $d^a=h$, hence 
$h \in H^a$, so $H \leq H^a$.

Conversely, 
suppose $a \in N_G(H)$. We want to show that $a \in N_G(\bigcap_{g\in
N_G(H)} D^g)$. So let $b \in \bigcap_{g\in N_G(H)} D^g$.  
For any $g\in N_G(H)$, since 
$ga^{-1} \in N_G(H)$, one has 
$b \in D^{ga^{-1}}$.  
Hence $b^a \in {(D^{ga^{-1}})}^a=D^g$, and 
$b^a \in \bigcap_{g\in N_G(H)} D^g$.  As 
$\bigcap_{g\in N_G(H)} D^g\leq (\bigcap_{g\in N_G(H)} D^g)^a$, 
we conclude that 
$a$ normalises $\bigcap_{g\in N_G(H)} D^g$, completing the
proof of the claim.\end{proof}
The claim shows that $N_G(H)$ is relatively definable. 
It is clear that the conjugate $H^g$ is relatively definable. 
We now prove the relative definability of the centralizer $C_H(A)$.  

Note that $C_H(A)=C_{X}(A)\cap H$. But
As $H$ is definable relative to $G$, $H=D\cap G$
for some definable subgroup $D$ of $X$. Let $S:=C_{X}(A)\cap D$.  
As $X$ is icc, $C_{X}(A)$ is definable in 
$X$.  So $S$ is definable in $X$. 
But $C_H(A)=C_{X}(A)\cap H=S\cap G$. So $C_H(A)$ is relatively
definable.\end{proof}
\begin{lemma} Let $G$ be an icc group.\itemize
\item Let $S$ be a subset of $G$. 
Then $S$ satisfies the uniform chain
condition on intersections with $S$ of chains of intersections of
families of uniformly definable subgroups of $G$.  
\item Let $H$ be a subgroup of $G$, 
and $K$ a relatively definable normal subgroup of $H$. Then $H/K$ is
naturally isomorphic to a sub-icc group, and the isomorphism preserves
relatively definable subgroups.\end{lemma}
\begin{proof} Let $\{H_i\cap S\}_{i\in I}$ 
be a chain (ascending or descending), where each $H_i$ is 
uniformly definable. Then for each $i$, the intersection 
$\bigcap_{j\ge i}H_j$ is equal to a finite sub-intersection, and hence
is uniformly definable. Thus $\{\bigcap_{j\ge i}H_j\}_{i\in I}$ is 
of finite length bounded by a positive
integer $n$ (which depends only on the defining formula).

For the second part, let $T$ be a definable subgroup of $G$ such 
that $T\cap H=K$. By icc, the intersection 
$J:=\bigcap_{g\in H}T^g$ is a finite
sub-intersection, and hence definable. By Lemma 2.1, 
$N_G(J)$ is definable. Hence $N_G(J)$ is icc. But 
$H\cap J=K$ and $H\le N_G(J)$. Thus $H/K$ is isomorphic to
$HJ/J$, which is a subgroup of $N_G(J)/J$. But $J$ is definable, 
so $N_G(J)/J$ is icc because the icc is preserved under
definable quotients. Thus $H/K$ is isomorphic to a sub-icc group by 
an isomorphism which preserves relatively definable subgroups.
\end{proof}
The sub-icc chain condition and substability are 
preserved under arbitrary subgroups. Taking $S$ in Lemma 2.2
to be a subgroup of $G$, we see that a sub-icc group 
satisfies the uniform chain
condition on chains of intersections of families of uniformly relatively
definable subgroups. Lemma 2.2 implies that 
the sub-icc chain condition is preserved by
quotients by relatively definable normal subgroups. The same proof 
replacing a sub-icc group by a substable group shows that
substability is preserved under quotients by relatively definable
normal subgroups. 
\begin{lemma}\itemize
\item Let $G$ denote an icc-group and $A$ a subset of $G$.  
Then the iterated centralizers $C_G^i(A)$ and the terms in the 
upper central series $Z_i(G)$, where $i<\omega$, are definable.
\item Let $G$ denote a sub-icc group, $A$ a subset of $G$ and $H$ a
relatively definable subgroup of $G$.  Then, for $i<\omega$, 
the iterated centralizers
$C_H^i(A)$ and the upper central series $Z_i(G)$ are 
definable relative to $G$.\end{lemma}
\begin{proof} We prove the second part. The first
part is similar.  We use 
induction on $i$. The case $i=1$ is in Lemma 2.1.  
So we may assume that the result is true for $i-1$.  But
$$C^i_H(A)=C_H(A/C_H^{i-1}(A))\cap \bigcap_{j\in \{0,\dots,i-1\}}
N_H(C_H^j(A)).$$
By induction hypothesis, for all $j\in \{0,\dots,i-1\}$, the
iterated centralizers $C_H^j(A)$ are relatively definable. By Lemma 2.2, 
the quotients $G/C_H^j(A)$ for $j\in \{0,\dots,i-1\}$ are
sub-icc groups.  

By Lemma 2.1 (applied in the
sub-icc group $G/C_H^j(A)$), the centralizer 
$C_H(A/C_H^{i-1}(A))$ is relatively
definable. Also the normalizers $N_H(C_H^j(A))$ 
are relatively definable where $1\leq j\leq i-1$.  
Thus $C_H^i(A)$ is relatively definable. 
The finite terms of the upper
central series of $G$ are now relatively definable
as a series of iterated centralizers.\end{proof}

{\bf Remark} 
By Lemmas 2.1-2.3, a sub-icc group and all its
quotients by relatively definable normal subgroups satisfy 
the uniform chain condition on (iterated) centralizers, i.e.\ there
is a bound on the lengths of chains of (iterated) 
centralizers.  It is not hard to give examples of groups $G$ with the chain condition
on centralizers such that the quotient group $G/Z(G)$ has an infinite
descending chain of centralizers, or arbitrarily
long chains of centralizers.

We now consider solubility and nilpotency. 
We first extend [6], Corollary 2.3, 
in the case of sub-icc groups and obtain bounds
on the derived lengths involved. 
\begin{lemma} A locally nilpotent 
sub-icc group $G$ is soluble of
derived length bounded by the maximal 
length $k$ of all chains of centralizers. Moreover, $G$ has a series of  
characteristic relatively definable subgroups of length bounded by
$k$ with abelian factors. 
Each of these characteristic relatively definable subgroups can
be chosen to be of the form of the 
centralizer of a (finite) subset of $G$.\end{lemma}
\begin{proof} For the first assertion 
we can either use induction on $k$ or do a direct
construction of an infinite descending chain of centralizers. To do
induction on $k$ note that if $k=1$, then $G$ must be abelian, so the
result holds. Thus we may assume that the result is true for any
proper subgroup whose chains of centralizers have length bounded by
$k-1$. The proof of \cite{bryant},Corollary 2.3 
yields an element $z\in Z_2(G)-Z(G)$
such that $G/C_G(z)$ is abelian. By induction
hypothesis $C_G(z)$ 
is soluble of derived length bounded by $k-1$. Hence 
$G$ is soluble of derived length bounded by $k$.

To prove the last assertion consider the derived group
$G'$. The proof of [1], Corollary 2.3 shows that $G'\leq C_G(z)$ for any
element $z\in Z_2(G)-Z(G)$, and that the quotient group
$Z_2(G)/Z(G)$ is non-trivial (so there exists such an element $z$).  
Let $C_z:=C_G(z)$, where $z$ is any
element in $Z_2(G)-Z(G)$.  As $k-1$ bounds the lengths of chains of
centralizers in $C_z$, by the first assertion in this proposition, 
we see that $C_z$ is soluble of
derived length $k-1$.  But $G'\leq \bigcap_{z\in Z_2(G)-Z(G)} C_z$.

Let $I:=\bigcap_{z\in Z_2(G)-Z(G)} C_z$. Then $I$ is equal to a
finite sub-intersection $C_{z_1}\cap \dots \cap C_{z_n}$ for some
elements $z_1,\dots,z_n$ by the sub-icc chain condition.  Thus $I$ is
of the form of the centralizer $C_G(F)$ of a finite set $F$ (in
particular is definable relative to $G$), and is also supergroup of
$G'$ which is characteristic in $G$. Furthermore, $I$ is soluble of
derived length $k-1$. By either using induction on $k$ or by 
applying the same argument to
any $G^{(i)}$, for $i\in \{0,\dots,k\}$, we conclude that $G$ has a
series $\{1\}=G_0\le G_1\le\dots\le G_k=G$ of characteristic subgroups
such that each $G_i$ 
is definable relative to $G$ and is of the form of the centralizer 
of a subset of $G$ and such that 
all the quotients $G_{i+1}/G_i$ are abelian.  
This completes the proof.\end{proof}
\begin{lemma}\itemize
\item Let $G$ denote an icc-group. Let $H$ be a soluble
(resp.\ nilpotent) subgroup of $G$. Then
there is a definable soluble (resp.\ nilpotent) supergroup $D$ of $H$ 
of the same derived length (resp.\ nilpotency class), and whose normalizer
contains the normalizer of $H$. The 
defining formula for $D$ only depends on the derived length
(resp.\ class).
\item Let $G$ denote a sub--icc group. Let $H$ be a soluble
(resp.\ nilpotent) subgroup of $G$. Then
there is a relatively 
definable soluble (resp.\ nilpotent) supergroup $D$ of $H$ 
of the same derived length (resp.\ nilpotency class), and whose normalizer
contains the normalizer of $H$. The relative 
defining formula for $D$ only depends on the derived length
(resp.\ class).\end{lemma} 
\begin{proof} We use induction on the derived length of $H$ in the
soluble case.  
Consider the last nontrivial commutator subgroup $A$ of $H$. Then
$B:=Z(C_G(A))\geq A$ is definable and abelian.  But $N_G(B)$
contains $H$, so by Lemma 2.1, 
$N_G(B)/B$ is a definable group containing $HB/B$, which has
smaller derived length. Inductively, we find a 
definable supergroup $T/B$ of $HB/B$ in $N_G(B)/B$ whose defining
formula only depends on its derived length and which is (definably) 
soluble of the same length as $HB/B$. We let $D$ be the pre-image $T$; and
the proof is complete in the soluble case.   

We now deal with the nilpotent case. We use 
induction on the nilpotency class of $N$. If it is zero,
then $N=(1)$, and the result is trivial, so we may assume that the class is at 
least $1$ and
that it holds for smaller class. Consider the definable (in 
$G$ -- use Lemma 2.1) groups 
$Z(C_G(Z(N))$ and $C_G(Z(N))$. The factor 
$Q:=C_G(Z(N))/(Z(C_G(Z(N))))$
contains the factor $F:=N/(Z(C_G(N)))$. Notice that $F$ is icc by 
Lemma 2.2. 
But $F$ has smaller nilpotency class than $N$. By induction 
hypothesis, $F$ is
contained in a nilpotent definable (in $Q$) subgroup 
$D^*:=D/(Z(C_G(Z(N))))$
of the same class as that of $F$ (which means that one class 
less than
the class of $N$). The pre-image $D$ is definable in 
$L:=C_G(Z(N))$. But $L$ is
definable in $G$. Thus $D$ is definable in $G$. On the other hand 
the nilpotency class of $D$ equals $1+k=1+k'$, 
where $k$ is the nilpotency class of $D^*$, 
and $k'$ the nilpotency class of $F$. This equals the nilpotency 
class of $N$; and we are done.

The second part now follows immediately. For if $G$ is sub-icc with
icc-supergroup $X$. Then by the first part $H$ is contained in a
group $D_1$ which is definable in $X$ and nilpotent (resp.\ soluble)
of the same nilpotency class (resp.\ derived length) as that of $H$.  
Then $D_1\cap G$ contains
$H$ and is definable relative to $G$ and has the required
properties.\end{proof}

\begin{lemma}
Let $G$ denote a periodic 
locally soluble sub-icc group.  
Then $G$ contains a non-trivial abelian normal subgroup.\end{lemma}

\begin{proof} Consider a countable elementary sub-model $G^*$ of
$G$.  As $G^*$ is locally finite, we have 
$G^*=\bigcup_{i\leq \omega} G_i$ where $G_i$ are finite soluble groups.  By
Lemma 2.5, 
there are non-trivial 
relatively definable abelian normal subgroups $A_i$ of
$G_i$ (take the relatively definable supergroup of the last
non-trivial term of the derived series of $A_i$). Each 
$H_j:=\bigcap_{i>j}N_{G^*}(A_i)$ is equal to a finite sub-intersection by
the icc chain condition, and hence is relatively definable.  Hence the family $\{H_j\}_{j<\omega}$  
consists of ascending and uniformly definable subgroups. Therefore,
by the icc chain condition, for some $i_0$ we have
$H_{i_0}=H_j$, for all $j>i_0$. Thus $A_{i_0}$ is a non-trivial 
normal abelian subgroup of $G^*$ since $H_j\geq G_j$ for any
$j<\omega$. Then $Z(C_{G^*}(A_{i_0}))$ yields 
a non-trivial definable abelian
normal subgroup of $G$.\end{proof}

We shall be using the easy observation that if 
$G$ is a locally finite group and $K$ a
normal subgroup such that $G/K$ is a countable $p$-group for some
prime $p$. Then there exists a
$p$-subgroup $P$ of $G$ such that $G=KP$. In particular, 
any countable $p$-subgroup of any quotient of $G$ 
is the homomorphic image of a $p$-subgroup of $G$.

\begin{lemma} 
Let $G$ denote a countable periodic locally soluble sub-icc group. 
Let $P$ be a maximal $p$-subgroup
of $G$ and $A$ a locally nilpotent normal subgroup of $G$. 
Then the Sylow-$p$ subgroups of $PA$ are conjugate to $P$ in $PA$, and
$PA/A$ is a Sylow $p$-subgroup of $G/A$. Let $p$ denote an odd prime. 
Then every $p$-subgroup of the  
factor $G/O_{p}(G)$ satisfies the descending chain condition 
on subgroups.\end{lemma}
\begin{proof} The first two statements follows from [5] (applied to mc-groups), 
the last part follows from the first two parts and 
thmrem A in [4].\end{proof}

\begin{lemma} Let $G$ be a periodic group and $A$ be an abelian
normal subgroup
whose $p$-parts $A_p$ satisfy the descending chain condition on
subgroups for all primes $p$.  
Then $A$ has only finitely many elements of any given
finite order.\end{lemma}
\begin{proof} See [5].\end{proof}
We shall also use the following. 

\begin{lemma} Let $p$ denote a prime and 
$G$ a locally finite 
group whose $p$-subgroups satisfy the descending chain condition
on subgroups.  
Then every $p$-subgroup of every 
homomorphic image of $G$ satisfies the descending chain condition
on subgroups.\end{lemma}
\begin{proof} See [5].\end{proof}

\begin{thm} Let $G$ be a countably infinite locally
finite group with a series of subgroups $\{T_i:i<\omega\}$ each normal
in its successor whose factors are finite.  Assume
that $G$ contains a finite Sylow $p$-subgroup for some prime $p$.  Then $G$
contains a normal $p'$-subgroup of finite index.\end{thm}
\begin{proof}
See [2].\end{proof}

\begin{thm} Let $G$ be a locally finite group.\itemize
\item $G$ contains an infinite abelian subgroup.
\item If $G$ is locally soluble, then the chief factors of $G$ are
elementary abelian groups.
\item If $G$ is locally nilpotent and satisfies the descending chain
condition on subgroups, then $G$ is virtually abelian.
\end{thm}
\begin{proof} The first part is a well-known result of 
Hall--Kulatilaka and Kargapolov.  
The second part is Corollary 1.B.4 in [5].  The third part is
Proposition 1.G.14 in [5].\end{proof}

{\bf Remark} This theorem 
immediately implies that a locally soluble locally finite simple group
must have prime order. Consequently, a periodic locally soluble group
cannot involve (i.e.\ have a section isomorphic to) 
an infinite simple group.

\begin{thm} Let $G$ denote a locally finite group and $p$ a
prime. Assume that the 
$p$-subgroups of $G$ satisfy the descending chain condition on subgroups.
Then the factor $G/O_{p'p}(G)$ is finite if and only if $G$ does not
contain any section which is isomorphic to the an infinite simple
group containing elements of order $p$.\end{thm}
\begin{proof} See [5].\end{proof}
 
If $G$ is locally (finite-soluble) and $p$ a prime such that the
$p$-subgroups of $G$ satisfy the descending chain condition on
subgroups, then the normal subgroup $O_{p'p}(G)$ has finite index in
$G$ by theorems 2.1.10 an 2.1.11.

Recall that a group $G$ is callled Cernikov if it is a
virtually abelian group satisfying the descending chain condition on
subgroups. Examples include groups with a finite series
whose factors are either finite or Prufer groups.  By a celebrated 
result of Cernikov [2], finite direct products of
finite cyclic groups and Prufer groups form precisely the class soluble
Cernikov groups.

\begin{thm} Let $G$ be a Cernikov group. Then $G$ is periodic, has only finitely many elements of any given
finite order, and contains a minimal characteristic 
subgroup of finite index which is divisible abelian.
\end{thm}
\begin{proof} The second part follows from the first part as
an abelian group with the descending chain condition on subgroups has
only finitely many elements of any given finite order; for details and
for the first and third parts see [5].\end{proof}

The centralizer-connected component $G^{cc}$ of an mc-group $G$ 
is the intersection of all centralizers of finite index. 
By the descending chain condition on centralizers, 
$G^{cc}$ is equal to a finite sub-intersection, 
hence is definable and has finite index
in $G$.

\begin{lemma} Let $G$ be a periodic locally soluble sub-icc group 
such that for any prime $p$, any $p$-subgroup of $G$ satisfies the 
minimal condition on subgroups. Then $G$ is virtually abelian.
\end{lemma}
to prove this We first prove the following.
\begin{claim} If every proper centralizer in $G$ is virtually
abelian, then $G$ is virtually abelian.\end{claim}
\begin{proof} 
Suppose that every proper subgroup of $G$ which is of the
form of the centralizer of some subset of $G$ is virtually abelian.
Note that $G^{cc}$ 
has finite index in $G$ by the descending chain
condition on centralizers (so in particular $G^{cc}$ is infinite).  
If $G^{cc}$ is a proper subgroup of $G$,
then it is virtually abelian by assumption. So we may assume that
$G^{cc}=G$, i.e.\ $G$ is centralizer-connected. Thus 
$G$ has no proper non-trivial centralizer of finite index.  

First we observe that $G$ has non-trivial centre.  For, as a periodic
locally soluble sub-icc group, by Lemma 2.6, $G$ 
contains a non-trivial normal
abelian subgroup $A$. Choose an arbitrary non-trivial element
$g$ in $A$. For any prime $p$, the $p$-parts of $A$ satisfies
the descending chain condition on subgroups. By 
Lemma 2.8, the abelian group 
$A$ has only finitely many elements of any given finite order (note that 
the fact applies to finite index subgroup). 
Thus $g$ has only finitely many $G$-conjugates, whence $C_G(g)$ has finite
index in $G$.  Since $G$ is centralizer-connected, we deduce that $g\in
Z(G)$ as required.

Now we consider the factor 
$G/Z(G)$. Let $H/Z(G)$ denote the centralizer-connected component of
the quotient $G/Z(G)$. So $H$ has finite index in $G$ by the
descending chain condition on centralizers. Clearly to prove that $G$
is virtually abelian, it suffices to prove that $H$ is virtually
abelian, and this is what we now do.

The factor $H/Z(G)$ is a periodic locally soluble sub-icc group. By Lemma 2.6, $H/Z(G)$ contains a non-trivial 
normal abelian subgroup $J/Z(G)$. \end{proof}

We now claim the following.
\begin{claim} $J$ is contained in the second iterated centre
$Z_2(H)$ of $H$.\end{claim}
The $p$-parts of $J/Z(G)$ satisfy the descending
chain condition on subgroups (since this holds for $J$ and is inherited by homomorphic images). 
By Lemma 2.8, the abelian group
$J/Z(G)$ has only finitely many elements of any given finite
order (again the fact applies to a finite index subgroup). 
Choose an arbitrary non-trivial 
element $aZ(G)\in J/Z(G)$. Then 
since $J/Z(G)$ is normal in $G/Z(G)$, we infer that 
$aZ(G)$ has only finitely many 
$H/Z(G)$-conjugates. Hence the centralizer 
$C_{H/Z(G)}(aZ(G))$ has finite index in $H/Z(G)$. Since $H/Z(G)$ is
centralizer-connected, we deduce that $aZ(G)\in Z(H/Z(G))$ which shows that
$J$ is contained in the second centre $Z_2(H)$. This proves Claim 3.

We now complete the proof of the Claim 2.  Claim 3 shows 
that commutation $f_a:=(x\rightarrow [x,a])$ with an element $a\in
J$ induces a homomorphism from $H$ into the centre $Z(H)$. Note that 
the kernel
$\ker(f_a)$ equals $C_H(a)$. Since $J/Z(G)$ is non-trivial, there is
an element $aZ(G)\in J/Z(G)$. So $a\in J-Z(G)$. Therefore $C_G(a)$
is proper in $G$, and hence is virtually abelian. So $C_H(a)$ is also
virtually abelian because $C_H(a)=C_G(a)\cap H$. 
Now we consider the image $f_a(H)$ which is contained in the centre
$Z(H)$. The $p$-parts of the centre $Z(H)$ 
satisfy the descending chain condition on
subgroups. Therefore by Lemma 2.8 , $Z(H)$ and
hence $f_a(H)$ have only finitely many elements of any given finite
order. 
Now there is an $n_a$ such that $a^{n_a}=1$. But as $a$ is in
the second iterated centre of $H$, we deduce that $[h,a]$ is central
in $H$ for any $h\in H$. 
In particular $[h,a]$ commutes with $a^j$ for any
$j<\omega$ (for any $h\in H$). Therefore for any $h\in H$ we have 
$$1=[h,a^{n_a}]=[h,a(a^{{n_a}-1})]=[h,a^{{n_a}-1}][h,a][h,a,a^{{n_a}-1}]=
[h,a^{{n_a}-1}][h,a].$$  
Repeating this $n_a-1$ many times we deduce that
$1=[h,a^{n_a}]=[h,a]^{n_a}$. Therefore 
the image $f_a(H)$ has finite exponent $n_a$. But 
$f_a\in Z(H)$, and $Z(H)$ has only finitely many elements
of any given finite order. Therefore the image $f_a(H)$ must be
finite. This shows that $H$ (and hence $G$) 
is virtually abelian, completing the
proof of Claim 2.
We now complete the proof of the lemma.  Consider the family
$\mathcal{F}:=\{C_i:i\in I\}$ consisting of 
subgroups of $G$ each of which is not virtually abelian and has  
the form of the centralizer of some subset of
$G$. We argue by contradiction to show that $\mathcal{F}$ is empty.
Suppose that $\mathcal{F}$ is non-empty. Then 
by the descending chain condition on centralizers, it contains
a minimal element, say $C_{i_0}=C_G(A_{i_0})$. Any proper centralizer
in $C_{i_0}$ has the form 
$C_{C_{i_0}}(X)=C_G(X)\cap C_{i_0}=C_G(X)\cap
C_G(A_{i_0})=C_G(X\cup A_{i_0})$ for some non-central subset $X$ of
$C_{i_0}$.  Therefore, $C_{C_{i_0}}(X)$ is virtually
abelian because of the minimality of $C_{i_0}$. So every proper centralizer
in $C_{i_0}$ is virtually abelian. Thus $C_{i_0}$ is virtually
abelian by Claim 2; which is a contradiction.

This contradiction shows that $\mathcal{F}$ is empty. Thus 
$G$ (which equals $C_G(1)$) is virtually abelian.

\begin{thm} Let $G$ be a periodic sub-icc group. Suppose that any two elements of $p$-power order generate a finite subgroup. Then the Sylow $p$-subgroups of $G$ are conjugate.\end{thm}
\begin{proof} This result is proven by Poizar-Wagner in the case of substable groups (see [6]), but the same proof goes through in the sub-icc case.\end{proof}

\section{Proof of thmrem 1.1}

For the first part note that 
the Hirsch--Plotkin radical $HP(G)$ is a locally nilpotent
mc-group. Hence $HP(G)$ is soluble by Lemma 2.4. 
By Lemma 2.5, there is a relatively
definable soluble supergroup $D$ of $HP(G)$ of the
same derived length as that of $HP(G)$ whose normalizer contains the
normalizer of $HP(G)$. So $D$ is a normal subgroup of $G$.  
By Lemma 2.2(2), the factor $G/D$ is a 
periodic locally soluble sub-icc group. Now by thmrem 2.14, 
for any prime $p$, the Sylow $p$-subgroups of
$G$ are conjugate as $G$ is locally finite and sub-icc.  
Since a locally finite group is soluble if and only if all its countable subgroups are soluble, we may assume that $G$ is countable, and by Lemma 2.7 
for any odd prime $p$, 
every $p$-subgroup of the factor $G/O_{p}(G)$ satisfies the descending chain 
condition on subgroups (note that 
this only holds for odd primes). Thus the 
homomorphic image $G/D$ also shares the same property.

Let $p$ denote an odd prime.  
The $p$-subgroups of $G/O_{p'}(G)$ also
satisfy the descending chain condition on subgroups.  
Therefore, the largest normal $p$-subgroup $O_{p'p}(G)/O_{p'}(G)$ of
$G/O_{p'}(G)$ satisfies the descending chain condition on subgroups.  
$O_{p'p}(G)/O_{p'}(G)$ is the
homomorphic image of a $p$-subgroup $P$ of $G$.  But $P$
satisfies the minimal condition on subgroups. Applying thmrem 2.13 we
deduce that $P$ is virtually abelian. Hence 
$O_{p'p}(G)/O_{p'}(G)$ is virtually abelian. Therefore, the 
finite index abelian normal subgroup $A$ of $O_{p'p}(G)/O_{p'}(G)$ 
is a Cernikov group. Hence, by thmrem 2.12 (3), $A$ contains 
a minimal divisible abelian characteristic (in $A$) 
subgroup, say $H_p/O_{p'}(G)$, of finite index which is clearly also 
normal and of finite index in $O_{p'p}(G)/O_{p'}(G)$.  
But by thmrem 2.11, 
$O_{p'p}(G)$ has finite index in $G$, and this implies that the
pre-image $H_p$ is a normal subgroup of finite index in $G$ such that 
$H_p/O_{p'}(G)$ is
divisible abelian for any odd prime $p$.

Let $I$ denote the intersection of all the $H_p$'s, where
$p$ ranges over all odd primes.  Any $H_p$ has finite index in
$G$ and hence contains all the divisible subgroups of $G$.  So, all
the divisible subgroups of $G$ are contained in the intersection
$I$.  We now claim that this shows that all the Sylow $p$-subgroups
of $G/I$ are finite for any odd prime $p$.

Let $p$ be an odd prime. Let $P$ be a Sylow $p$-subgroup of $G/I$.  
$P$ is the homomorphic image of a Sylow $p$-subgroup 
$P^*$ of $G$.  

For any odd prime $p$, any Sylow $p$-subgroup of $G$ 
satisfies the descending chain condition on subgroups, hence is 
virtually abelian by thmrem 2.10(3) (as it is locally nilpotent 
since it is locally finite). So for any odd prime $p$, any Sylow 
$p$-subgroup of $G$ is Cernikov, whence is divisible-by-finite by 
thmrem 2.12(3).  

So $P^*$ is divisible-by-finite. But $I$ contains all the divisible
subgroups of $G$. Hence $P^*\cap I$ has finite index in $P^*$.  
As $P$ equals $P^*I/I$, it must be finite.  
Therefore, for any odd prime $p$, any Sylow
$p$-subgroup of $G/I$ is finite.

The intersection over all odd primes $\bigcap_{p\neq 2} O_{p'}(G)$ 
is a normal $2$-subgroup of $G$ and hence is contained in the maximal
normal $2$-subgroup $O_2(G)$.  But for any odd prime $p$, the factor
$H_p/O_{p'}(G)$ is abelian.  Thus the derived group $I'$ is contained
in $\bigcap_{p\neq 2} O_{p'}(G)\leq O_2(G)$.  Therefore the factor
$I/O_2(G)$ is abelian.  The
$2$-group $O_2(G)$ is a locally nilpotent mc-group.  So 
$O_2(G)$ is soluble by
Proposition 2.4. 
Thus $I$ is soluble.  By Lemma 2.5, 
$I$ is contained in a relatively definable soluble subgroup $W$ (of
$G$) such that $N_G(I)\leq N_G(W)$. The
quotient $G/W$ is periodic locally soluble, and is also a sub-icc
group by Lemma 2.2(2).

As $G$ is locally (finite soluble), it must be a countable 
union of finite soluble subgroups $\bigcup_{i\in \omega}
S_i$. Thus for any prime $p$
and any $i$ there is a Sylow $p$-subgroup $S_{i,p}$ of $S_i$ (so-called 
Sylow basis) such that
for a any fixed $i_0$ the set of all $S_{i_0,p}$ for different primes $p$
form a Sylow system for the finite soluble group $S_{i_0}$ 
(i.e.\ $S_{i_0}$ is
generated by the Sylow $p$-subgroups $S_{i_0,p}$, and for different primes
$p_1$ and $p_2$ we have
$S_{i_0,p_1}S_{i_0,p_2}=S_{i_0,p_2}S_{i_0,p_1}$), and for a fixed
prime $p_0$ the set of all $S_{i,p_0}$ is an ascending chain whose
union $\widehat{S_p}:=\bigcup_{i\in \omega} S_{i,p}$ is a 
Sylow $p$-subgroup of $G$.  
Hence, for different primes $p_0$ and $p_1$ we have
$S_{p_1}S_{p_2}=S_{p_2}S_{p_1}$. 
Each $S_i$ is generated by its (pairwise ``permutable but not commuting'') 
Sylow $p$-subgroups $S_{i,p}$, for 
all primes $p$. Consequently, $G$ is generated by its pairwise
permutable Sylow $p$-subgroups $S_{p}$, for all primes $p$.  

Thus, we can write $G$ as an ascending union (over
all the primes $p$) of the
subgroups $U_p$, where $U_p$ is the subgroup generated by all the
$G$-Sylow subgroups $\widehat{S_2},\widehat{S_3},\dots,\widehat{S_p}$.  
Each $U_p$ in turn can be
written as an ascending union (over all positive integers $i$) of the
subgroups $V_{i,p}$ where $V_{i,p}$ is the subgroup generated by all the
$S_i$-Sylow subgroups $S_{i,2},S_{i,3},\dots,S_{i,p}$. Note that each $V_{{i,p}}$ is a
$\{2,3,\dots,p\}$-group, and hence so is each $U_p$.  

Let $n$ denote a bound on the lengths of centralizers.  
We have $\widehat{S_2} \leq U_p \leq G$ for any $p$.  The
Sylow $2$-subgroup $\widehat{S_2}$ is a locally nilpotent sub-icc-group, and hence
is $n$-soluble by Lemma 2.4. 
For any prime $p$, the maximal normal $2$-subgroup $O_2(U_p)$ of 
$U_p$ is the
intersection of all the Sylow $2$-subgroups of $U_p$ 
Consequently, $O_2(U_p)$ is contained in
$\widehat{S_2}$ for any prime $p$. Hence $O_2(U_p)$ 
is $n$-soluble for any prime $p$.  

Let $\Pi$ denote the set of all primes. By Lemma 2.5, 
there is a uniformly relatively definable family of $n$-soluble
subgroups $\{D_p\}_{p\in \Pi}$, where $D_p$ contains $O_2(U_p)$ and
$N_G(D_p)\geq N_G(U_p)$. Hence, by
the icc chain condition, there are finitely many primes
$p_1,\dots,p_m$ such that the intersection $I:=\widehat{S_2}\cap
\bigcap_{p\in \Pi} D_p$ is equal to the finite sub-intersection 
$\widehat{S_2}\cap \bigcap_{p\in \{p_1,\dots,p_m\}} D_p$.  

Consider any $U_p$, and let 
$\{2,p_1,\dots,p_k\}$ be the prime divisors of the orders of its 
elements.  By the second part, the $p_j$-Sylow subgroups of $U_p$ are
finite for $j\in \{1,\dots,k\}$.  But, as $U_p$ is periodic locally
soluble, it cannot involve infinite simple locally finite groups 
by thmrem 2.10. 
Thus, $U_p$ has a series $\{T_i:i<\omega\}$ of
subgroups, such that $T_i$ is normal in $T_{i+1}$ and $T_{i+1}/T_i$ is
finite for any $i<\omega$. So by thmrem 2.9, 
$U_p$ contains a normal $p_{j}'$-subgroup $F_j$ of finite index.
Hence, $O_2(U_p)$ has finite index in $U_p$ since 
$U_p$ contains a normal $2$-subgroup of finite index, namely, 
the intersection $F_1\cap \dots \cap F_{p_k}$.  Therefore, each $D_p$
and thus $I$ have finite index in $\widehat{S_2}$.  

Now we can write
the normal closure $I^G$ as a union of normal closures 
$\bigcup_{p\in \Pi} I^{U_p}$. But for any prime $p$, 
since $O_2(U_p)$ is
normal in $U_p$ and $N_G(D_p)\geq N_G(U_p)$, we deduce that 
$I^{U_p}$ is contained in $D_p$, and so is 
$n$-soluble. As a union of $n$-soluble subgroups, 
$I^G$ is also $n$-soluble. By Lemma 2.5, there is a
relatively definable $n$-soluble subgroup $J$ containing
$I^G$.  But the factor $G/J$ is periodic locally soluble

sub-icc by Lemma 2.2.  By Lemma 2.7, the group 
$\widehat{S_2}J/J$ is a Sylow $2$-subgroup of $G/J$. By the sub-icc chain
condition, $I$ has
finite index in $\widehat{S_2}$, so $\widehat{S_2}J/J$ is finite.  
Therefore, all the other Sylow $2$-subgroups of $G/J$ are finite as well.

To complete the proof we note that as a 
homomorphic image of $G$ for odd
primes $p$, the $p$-subgroups of $G/J$ have the descending chain
condition on subgroups. But since the Sylow $2$-subgroups of $G/J$ are finite,
we conclude that for all primes $p$, the $p$-subgroups of $G/J$
have the descending chain condition on subgroups. Now 
$G/J$ is virtually abelian. The proof is complete.

%\bibliographystyle{acm}
%\bibliography{bibsol}

\

\centerline{References}

\

[1] Bryant, R.M. Groups with the minimal condition on centralizers, {\it J. Algebra 60}, 2 (1079), 371-383.

[2] Cernikov, S. N. Infinite locally finite groups with finite Sylow subgrous, {\it Mat. Sb. (N.S.) 52 (94)} (1960), 647-652.

[3] Derakhshan, J., and Wagner, F.O. Nilpotency in groups with chain conditions. {\it Quart. J. Math. Oxford Ser (2) 48}, 192 (1997), 453-466.

[4] Hartley, B. Sylow {\it p}-subgroups and local {\it p}-solubility. {\it J. Algebra 23} (1972), 347-369.

[5] Kegel, O.H., and Wehrfritz, B.A.F. {\it Locally finite groups.} North-Holland Publishing Co., Amsterdam-London; American Elsevier Publishing Co., Inc. New York, 1973. North-Holland Mathematical Library, Vol. 3.

[6] Wagner, F.O. {\it Stable groups}, London Mathematical Society Lecture Note Series, {\it 240}, 
Cambridge University Press, Cambridge, (1997).

\end{document}